\documentclass[a4paper,11pt]{amsart}
\usepackage{amsmath}
\usepackage{amssymb}
\usepackage{amsfonts}
\usepackage{mathtools}
\usepackage{amscd}
\usepackage{latexsym,graphicx}
\input xy
\xyoption{all}
\usepackage[latin1]{inputenc}
\usepackage[cyr]{aeguill}
\usepackage{xspace}
\usepackage{array}
\usepackage{newclude}
\usepackage{hyperref}
\hypersetup{pdfpagelabels,
                  plainpages=false,
                    colorlinks=true,       
                    linkcolor=red}
\renewcommand*{\HyperDestNameFilter}[1]{\jobname-#1} 
\numberwithin{equation}{section}
\usepackage[centering]{geometry}

\usepackage[nameinlink]{cleveref}

\makeatletter
\providecommand{\leftsquigarrow}{%
  \mathrel{\mathpalette\reflect@squig\relax}%
}
\newcommand{\reflect@squig}[2]{%
  \reflectbox{$\m@th#1\rightsquigarrow$}%
}
\makeatother

\usepackage{latexsym,graphicx}
\usepackage{mathtools}
\usepackage{tensor}

\newcommand{\sspace}{\vspace{0.25cm}}

\theoremstyle{plain}
\newtheorem{theor}{Theorem}[section]
\newtheorem{conj}[theor]{Conjecture}
\newtheorem{prop}[theor]{Proposition}
\newtheorem{lem}[theor]{Lemma}

\newtheorem{cor}[theor]{Corollary}

\theoremstyle{remark}
\newtheorem{rem}[theor]{Remark}

\newtheorem{Example}[theor]{Example}

\newtheorem{Question}[theor]{Question}

\theoremstyle{plain}

\newtheorem{defi}[theor]{Definition}

\newcommand{\CC}{{\mathbb C}}
\newcommand{\RR}{{\mathbb R}}
\renewcommand{\SS}{{\mathbf S}}
\newcommand{\QQ}{{\mathbb Q}}

\newcommand{\ZZ}{{\mathbb Z}}

\newcommand{\G}{{\mathbf G}}
\newcommand{\HH}{{\mathbf H}}

\newcommand{\PP}{{\mathbf P}}

\newcommand{\M}{\mathbb{M}}

\newcommand{\UU}{{\mathbf U}}

\newcommand{\NN}{{\mathbb N}}

\newcommand{\AAf}{\mathbb{A}_{\rm f}}

\newcommand{\Ga}{\Gamma}

\newcommand{\Gr}{{\textnormal{Gr}}}

\newcommand{\ti}[1]{\mbox{$\tilde{#1} $}}

\newcommand{\ol}{\overline}

\newcommand{\lo}{\longrightarrow}

\newcommand{\sm}{{\rm sm}}

\newcommand{\End}{{\rm End}\,}

\newcommand{\Res}{{\rm Res}}
\newcommand{\Sh}{{\rm Sh}}

\newcommand{\Gm}{{\mathbf{G}_{m}}}

\newcommand{\X}{{\mathcal X}}

\newcommand{\He}{{\mathcal H}}

\newcommand{\der}{{\rm der}}

\newcommand{\GL}{{\rm \bf GL}}

\newcommand{\MT}{{\mathbf{P}}}

\newcommand{\g}{{\mathfrak{g}}}

\newcommand{\proj}{{\mathbf P}}

\newcommand{\rk}{\textnormal{rk}\,}

\newcommand{\Aut}{\textnormal{Aut}}
\newcommand{\an}{\textnormal{an}}

\newcommand{\HHom}{\textnormal{Hom}}
\newcommand{\cd}{\textnormal{cd}}

\newcommand{\cC}{{\mathcal C}}

\newcommand{\cD}{{\mathcal D}}

\newcommand{\cX}{{\mathcal X}}

\newcommand{\AAA}{{\mathcal A}}

\newcommand{\lto}{\longrightarrow}

\newcommand{\Def}{\textnormal{Def}}

\newcommand{\GmC}{\mathbf{G}_{\textnormal{m}, \CC}}
\newcommand{\SL}{{\mathbf{SL}}}

\newcommand{\Lie}{{\rm Lie}\,}

\newcommand{\Imm}{{\rm Im}}

\newcommand{\Zar}{\textnormal{Zar}}

\newcommand{\Id}{\textnormal{Id}}
\newcommand{\VV}{\mathbb{V}}
\newcommand{\VVV}{\mathcal{V}}
\newcommand{\HL}{\textnormal{HL}}
\newcommand{\OO}{\mathcal{O}}

\newcommand{\im}{\textnormal{Im}}
\newcommand{\Ad}{\textnormal{Ad}\,}

\newcommand{\Hod}{\textnormal{Hod}}

\newcommand{\Rep}{\mathbf{Rep}}
\newcommand{\MH}{\mathcal{MH}}
\newcommand{\W}{\mathbf{W}}
\newcommand{\Hcd}{\textnormal{H-cd}}
\newcommand{\atyp}{\textnormal{atyp}}
\newcommand{\NL}{\textnormal{NL}}
\newcommand{\w}{\mathfrak{w}}
\newcommand{\p}{\mathfrak{p}}

\begin{document}

\title{Hodge loci and atypical intersections: conjectures}
\author{B. Klingler}
\begin{abstract}
We present a conjecture on the geometry of the Hodge
locus of a (graded polarizable, admissible) variation of mixed Hodge structure over a complex
smooth quasi-projective base, generalizing to this context the Zilber-Pink
Conjecture for mixed Shimura varieties (in particular the Andr\'e-Oort
conjecture).
\end{abstract}
\maketitle

\tableofcontents

\section{Introduction: Hodge locus, atypical locus and main
  conjecture} \label{intro}

\subsection{Hodge locus}
The general context of this paper is the study of the following geometric
problem. Let $k$ be an algebraically closed field and let $f: \X \lo S$ be a smooth morphism
of quasi-projective $k$-varieties. Can
we describe the locus of closed points $s \in S$ where the motive $[\X_s]$ of the
fiber $\X_s$ is ``simpler'' than the motive of the fiber at a very general
point? Here ``simpler'' means that the fiber $\X_s$ and its powers contain more algebraic
cycles than the very general fiber and its powers. If a Tannakian formalism of $k$-motives
were available, this would be equivalent to saying that the motivic
Galois group $\textnormal{GMot}(\X_s)$ is smaller than the motivic
Galois group of the very general fiber.

We restrict ourselves to $k=\CC$. From now on all algebraic varieties are over
$\CC$. Following a common abuse of notation we will still denote by $S$ the
complex analytic space $S^\an$ associated to an algebraic variety
$S$; the meaning will be clear from the context. By a point of $S$ we
always understand a closed point.

We consider the Hodge incarnation of our
problem. Let $\VV \rightarrow S$ be a 
variation of mixed Hodge structures (VMHS)
over a smooth quasi-projective variety $S$. In this introduction we
will consider $\QQ$VMHS (we will restrict ourselves to
$\ZZ$VMHS when monodromy arguments are involved, as it simplifies the
exposition). The
weight filtration on $\VV$ is denoted by $W_\bullet$ and the Hodge filtration
on $\VV \otimes_{\QQ} \OO_{S}$ by $F^\bullet$. In this
paper all such variations are assumed to be {\it graded-polarizable and
admissible}. A typical example of such a gadget is $\VV=
R^mf_* \QQ$ for $f:
\X \lo S$ smooth algebraic, locally trivial for the usual topology. Precise definitions of
Hodge theory are recalled in \Cref{VMHS}. The non-expert reader will think of the
case of a (polarizable) variation of pure Hodge structures
(corresponding geometrically to the case where $f$ is moreover proper).

Replacing algebraic cycles by Hodge classes and motivated by the Hodge conjecture, one wants to understand
the Hodge locus $\HL(S, \VV) \subset S$, namely the subset of $S$ of  
points $s$ for which exceptional Hodge tensors for $\VV_{s}$ do occur. 

The Tannakian formalism available for Hodge structures is
particularly useful for describing $\HL(S, \VV)$. We recall it here,
as it will be crucial for the statement of our main conjecture. For
every $s\in S$, the Mumford-Tate group $\MT_s(\VV)$ of the Hodge structure
$\VV_{s}$ is the Tannakian group of the Tannakian category $\langle
\VV_{s}^\otimes\rangle$ of mixed $\QQ$-Hodge structures
tensorially generated by $\VV_{s}$ and its dual 
$\VV_{s}^\vee$. Equivalently, the group $\MT_s(\VV)$ is the stabiliser
of the Hodge tensors for $\VV_{s}$, i.e. the Hodge classes in the
rational Hodge structures tensorially generated by $\VV_{s}$ and
its dual. This is a connected $\QQ$-algebraic group, which is
reductive if $\VV_{s}$ is pure, and an extension of the reductive
group $\MT_s(\Gr^{W}_\bullet \VV)$ by a unipotent group in general
(where $W$ denotes the weight filtration on $\VV$). A point $s \in
S$ is said to be Hodge generic for $\VV$ if $\MT_s(\VV)$ is
maximal. If $S$ is irreducible, two Hodge generic points of $S$ have
the same Mumford-Tate group, called the generic Mumford-Tate group
$\MT_{S}(\VV)$ of $(S, \VV)$. The Hodge locus $\HL(S, \VV)$ is the subset of
points of $S$ which are not Hodge generic.

A fundamental result of Cattani-Deligne-Kaplan \cite{CDK95} (in the pure
case, extended to the mixed case in \cite{BPS})
states that $\HL(S, \VV)$ is a countable union of closed irreducible algebraic
subvarieties of $S$. A {\em special subvariety} of $(S, \VV)$ is by
definition an irreducible subvariety of $S$ maximal among the
irreducible subvarieties with a fixed generic Mumford-Tate
groups. Special subvarieties of
dimension zero are called {\em special points} of $(S, \VV)$. A
special point $s$ whose Mumford-Tate group $\MT_s(\VV)$ is a torus is
called a point with complex multiplication ({\em CM-point}) for $(S, \VV)$.

In a nutshell, we would like to address the following vaguely stated:
\begin{Question} \label{main question}
Given a smooth quasi projective variety $S$, any
variation of mixed Hodge structure $\VV \lo S$ produces naturally a
countable collection of irreducible subvarieties of $S$: its special subvarieties. Can one
describe the distribution of the special subvarieties strictly
contained in $S$, in
particular of the CM-points?
\end{Question}

\subsection{Zariski-closure of the Hodge locus} \label{Zariski}
A first precise version of \Cref{main question} would be to describe the
Zariski-closure of $\HL(S, \VV)$, in particular to answer the
following

\begin{Question} \label{Zariski-closure}
{\em Are there any geometric constraints on the Zariski-closure of $\HL(S,
\VV)$? Can one describe the couples $(S, \VV)$ such that $\HL(S, \VV)$
is Zariski-dense in $S$?} 
\end{Question}

Particular cases of this problem have been
classically considered by complex algebraic geometers, essentially
when $\VV$ is pure of weight $1$ or $2$, and using infinitesimal
methods which lead to density results even for the Archimedean topology:

\begin{Example} ($\VV$ pure of weight $1$) \label{Example1}
When $\VV$ is pure of weight $1$ (hence we are essentially considering
families of Abelian varieties), \Cref{Zariski-closure} has been raised
in a particular case in \cite{ColPi} and \cite{Iz}. 
A typical result is the following. Let
$S\subset \AAA_g$ be a subvariety of codimension at most $g$ of the moduli space $\AAA_g$ of
principally polarized Abelian varieties of dimension $g$. Then the
set $S_k$ of points $s \in S$ such that the corresponding Abelian
variety $A_s$ admits an Abelian subvariety of dimension $k$ is dense
(for the Archimedean topology) in $S$ for any integer $k$ between
$1$ and $g-1$. Let $\VV$ be the VHS restriction to $S$ of the Hodge incarnation $R^1 f_*\QQ$ of the universal Abelian variety
$f:\mathfrak{A}_g \lo \AAA_g$ over $\AAA_g$. As $S_k \subset \HL(S, \VV)$ it follows that the set $\HL(S, \VV)$ is dense in $S$. 
More generally let $(\G,
X)$ be a pure Shimura datum, $\Sh_K(\G, X)$ an associated Shimura variety,
$\HH \subset \G$ a $\QQ$-reductive subgroup and $S \subset \Sh_K(\G,
X)$ an algebraic subvariety. Let $\VV \lo S$ be the restriction to
$S$ of the variation of pure Hodge structures on $\Sh_K(\G, X)$
associated to any faithful algebraic representation of $\G$. Denote by
$\HL(S,\VV, \HH) \subset \HL(S, \VV)$ the subset of points $s \in S$ whose
Mumford-Tate group $\MT_s(\VV)$ is $\G(\QQ)$-conjugated to $\HH$. In
\cite{Chai} Chai defines an invariant $c(\G, X, \HH) \in \NN$, whose
value is $g$ in the example above, which
has the property that $\HL(S, \VV, \HH)$ is dense in $S$ as soon as $S$ has
codimension at most $c(\G, X, \HH)$ in $\Sh_K(\G, X)$.
\end{Example}

\begin{Example} ($\VV$ pure of weight $2$: Noether-Lefschetz locus) \label{Example2}
Let $S \subset \proj H^0(\proj^3_\CC, \OO(d))$ be the open subvariety moduli space of
smooth surfaces of degree $d$ in $\proj^3_\CC$. For $d>3$, the classical Noether theorem
states that the very general surface $Y \in S$ has Picard group $\ZZ$:
every curve on $Y$ is a complete intersection of $Y$ with another
surface in $\proj^3_\CC$. The countable union $\NL(S)$ of algebraic
subvarieties of $S$ corresponding to surfaces with bigger Picard group
is called the Noether-Lefchetz locus of $S$.
Let $\VV \lo S$ be the VHS $R^2f_*\QQ$, where $f: \mathcal{Y} \lo
S$ denotes the universal family. Clearly $\NL(S) \subset \HL(S,
\VV)$. In \cite{CHM} Ciliberto, Harris and Miranda proved that $\NL(S)$
is Zariski-dense in $S$. Green (see \cite[Prop.5.20]{Voisin}) proved the stronger result
that $\NL(S)$ is dense in $S$ for the Archimedean topology. In
particular $\HL(S, \VV)$ is dense in $S$.
\end{Example}

\Cref{Example1} and \Cref{Example2} indicate that special
subvarieties for $(S, \VV)$ are quite common in general. Even if they have a Hodge theoretic
significance, they are not special enough to force any global shape for the Zariski-closure of $\HL(S,
\VV)$. Hence we cannot expect a naive answer to the naive
\Cref{Zariski-closure}. 

\subsection{Atypical locus}
In this note we define a natural subset $S_\atyp(\VV) \subset\HL(S,
\VV)$: the {\em atypical locus} of $(S, \VV)$. While the
Zariski-closure of $\HL(S, \VV)$ can be wild, we conjecture (main
\Cref{main conjecture1}) that the structure of
$S_\atyp(\VV)$ is simple. This generalizes the Zilber-Pink Conjecture
for Shimura varieties to any (graded-polarized, admissible) $\QQ$VMHS over
any smooth quasi-projective base.

\subsubsection{Hodge codimension}
The crucial notion for defining the atypical locus $S_\atyp(\VV)
\subset S$ is the notion of Hodge codimension:

\begin{defi} (Hodge codimension) \label{Hcd}
Let $S$ be an irreducible quasi-projective variety and $\VV \lo S^\sm$ a variation
of mixed Hodge structures on the smooth locus $S^\sm$ of $S$. Let
$\PP_S$ be the generic Mumford-Tate group of $(S^\sm, \VV)$ and $\p_S$ its Lie
algebra (endowed with its canonical mixed $\QQ$-Hodge structure, of
weight $\leq 0$).

We define the Hodge codimension of $S$ with respect to $\VV$ as 
\begin{equation*} 
\Hcd(S, \VV) := \dim_\CC(  \Gr^{-1}_F (\p_S \otimes_\QQ \CC) ) - \rk \Imm \ol{\nabla}\;\;,
\end{equation*}
where $\ol{\nabla}: TS^\sm \lo \Gr^{-1}_{F} (W_0 \End \VV \otimes_\QQ \OO_S)$ is the
Kodaira-Spencer map (see \Cref{KS}) of $(S, \VV)$. 
\end{defi}

\begin{rem} \label{rem1}
It follows from \Cref{Hcd} that if $Y \subset Y' \subset S$ if a
pair of irreducible subvarieties and if $\PP_Y= \PP_Y'$ then $\Hcd(Y',
\VV_{|{Y'}^\sm}) \leq \Hcd(Y, \VV_{|{Y}^\sm})$.
\end{rem}

\subsubsection{Atypical subvarieties}

\begin{defi} (Atypical subvariety) \label{atypical}
Let $S$ be an irreducible smooth quasi-projective variety and $\VV \lo S$ a variation
of mixed Hodge structures on $S$.

An irreducible subvariety $Y \subset S$ is said to be atypical for
$(S, \VV)$ if 
\begin{equation}  \label{atypicality condition}
\Hcd(Y, \VV_{|Y^\sm}) < \Hcd(S, \VV) \;\;.
\end{equation}

We denote by $S_{\atyp}(\VV) \subset S$ the subset of $S$ given by the
union of
all atypical subvarieties for $(S, \VV)$.
\end{defi}

One easily checks that $S_\atyp(\VV)$ is contained in $\HL(S, \VV)$
(see \Cref{rem0}).

\subsection{Optimal subvarieties}
Let us introduce a notion cunningly different from atypicality:
\begin{defi} (Optimal subvariety) \label{optimal}
Let $S$ be an irreducible smooth quasi-projective variety and $\VV \lo S$ a variation
of mixed Hodge structures on $S$.

An irreducible subvariety $Y \subset S$ is said to be optimal for
$(S, \VV)$ if for any irreducible subvariety $Y \subsetneqq Y' \subset S$ containing
$Y$ strictly, the following inequality holds:
\begin{equation}  \label{atypicality condition}
\Hcd(Y, \VV_{|Y^\sm}) < \Hcd(Y', \VV_{|Y'^\sm}) \;\;.
\end{equation}
\end{defi}

Notice that if $Y \subset S$ is optimal for $(S, \VV)$
then $Y$ is atypical for $(S, \VV)$ and that, conversely, any
irreducible $Y \subset S$ which is atypical for $(S, \VV)$ {\em and maximal for this
property} is optimal for $(S, \VV)$.

\subsection{Statements of the main conjecture}
We have seen in \Cref{Zariski} that the Zariski-closure of the Hodge
locus $\HL(S, \VV)$ can be complicated. The main object of this text
is to present the following conjectures (shown in \Cref{On main} to be
equivalent), which predict that the subset $S_{\atyp}(\VV)$ of $\HL(S,
\VV)$ on the contrary has a simple geometry (or
equivalently that optimal subvarieties are rare):

\begin{conj} (Main conjecture, form $1$) \label{main conjecture1}
For any irreducible smooth quasi-projective variety $S$ endowed with a variation
of mixed Hodge structures $\VV \lo S$, the subset $S_{\atyp}(\VV)$ is a
finite union of special subvarieties of $S$.
\end{conj}

\begin{conj} (Main conjecture, form $2$) \label{main conjecture2}
For any irreducible smooth quasi-projective variety $S$ endowed with a variation
of mixed Hodge structures $\VV \lo S$, the subset $S_{\atyp}(\VV)$ is a
strict algebraic subset of $S$.
\end{conj}

\begin{conj} (Main conjecture, form $3$) \label{main conjecture3}
For any irreducible smooth quasi-projective variety $S$ endowed with a variation
of mixed Hodge structures $\VV \lo S$, the subset $S_{\atyp}(\VV)$ is
not Zariski-dense in $S$.
\end{conj}

\begin{conj} (Main conjecture, form $4$) \label{main conjecture4}
Any irreducible smooth quasi-projective variety $S$ endowed with a variation
of mixed Hodge structures $\VV \lo S$ contains only finitely
many irreducible subvarieties optimal for $(S, \VV)$.
\end{conj}

\subsection{Organization of the paper}
This note is organized as follows. \Cref{VMHS} provides a recollection on mixed Hodge
theory for the non-expert. \Cref{Mixed Hodge varieties} defines the notion of mixed Hodge
varieties (a generalization in the complex analytic category of mixed
Shimura varieties) and the corresponding period maps. Although most of
the material in this section reorganizes classical
results, our treatment is resolutely ``group-oriented'': Deligne's formalism of Shimura
data (and its generalisation to Hodge data) seems to offer an unrivaled 
functorial setting for Hodge theory. \Cref{geometric} explains
\Cref{main conjecture1} in terms of the geometry of period maps and
atypical intersection in the sense of \cite{Zannier}. In
\Cref{On main} we prove the equivalence of \Cref{main conjecture1}, \Cref{main conjecture2},
\Cref{main conjecture3} and \Cref{main conjecture4}, and explain the relation between these
conjectures and more classical statements like the Zilber-Pink
Conjecture for Shimura varieties (and its particular case the Andr\'e-Oort Conjecture). \Cref{CY}
details the simplest example of \Cref{main conjecture1} outside the
world of Shimura varieties: the case of Calabi-Yau $3$-folds. \Cref{functional
  transcendence} describes the relation between \Cref{main
  conjecture1} and a functional transcendence statement of Ax-Schanuel type
for period maps (\Cref{Ax-Schanuel}). 


\subsection{Acknowledgments} I would like to thank Emmanuel Ullmo for
his comments and his suggestion to add the interpretation of our main
conjecture in terms of optimal subvarieties (see \Cref{main
  conjecture4}). I thank the referees for their careful reading, in
particular for forcing me to clarify the (subtle) relation between my definition of mixed Hodge
datum and Pink's definition of mixed Shimura datum (see \Cref{ShimType}).

\section{Mixed Hodge theory}\label{VMHS}

\subsection{Deligne torus}
The Deligne torus is the restriction of scalars $\SS:=\Res_{\CC/\RR} \Gm$. So
$\SS_\CC$ is canonically isomorphic to $\GmC \times \GmC$ but
the action of complex conjugation on $\SS_\CC$ is given by the usual
one twisted by the exchange of the two factors. In particular
$\SS(\RR)= \CC^* \subset \SS(\CC) = \CC^* \times \CC^*$ consists of
the points of the form $(z, \ol{z})$.

Let $w: \mathbf{G}_{\textnormal{m}, \RR} \lto \SS$ be the
cocharacter whose values on real points is given by $\RR^* \subset
\CC^*$. We define the cocharacter $\mu: \GmC \lo \SS_\CC$ to be
the unique cocharacter such that $\overline{z} \circ \mu$ is trivial
and $z \circ \mu = \Id \in \End(\GmC)$, where $z$, $\ol{z}$ are
the two characters of $\SS$ generating its character group such that
the induced maps on points $\CC^* = \SS(\RR) \subset \SS(\CC) \lo
\Gm(\CC)= \CC^*$ are the identity, resp. complex conjugation. On
$\CC$-points, identifying $\SS(\CC)= \CC^* \times \CC^*$, we have
$\mu: \CC^* \lo \CC^* \times \CC^*$ given by $w \mapsto (w, 1)$.

\subsection{Pure Hodge structure}
We denote by $R$ one of the rings $\ZZ$, $\QQ$, $\RR$. 
Given $V$ an $R$-module we write $V_\RR:= V \otimes_R \RR$ and
$V_\CC:= V \otimes_R \CC$.

A pure $R$-Hodge structure (resp. of weight $n \in \ZZ$) is a Noetherian $R$-module $V$ together
with a morphism of algebraic groups $\varphi: \SS \lo \GL(V_\RR)$ (resp. such that $\varphi \circ w$ is given by $\CC^* = \SS(\RR)
\ni z \mapsto z^{-n}\cdot \Id_V$). Notice that if $R$ is a field, any
pure $R$-Hodge structure is a direct sum of pure $R$-Hodge
structure of fixed weight.

Equivalently, a pure $R$-Hodge structure of weight $n \in \ZZ$ is a
Noetherian $R$-module $V$ together with a bigraduation $V_\CC=  \bigoplus_{p+q=n} V^{p,q}$ satisfying $\ol{V^{p,q}} =
  V^{q,p}$, or a decreasing filtration called the Hodge filtration
  $F^p$ on $V_\CC$ such that $F^p \oplus \overline{F^{n-p+1}}
  \stackrel{\simeq}{\lo} V_\CC$. The equivalence between $\varphi$, the
  Hodge filtration and the bigraduation is as follows: the subspace
  $V^{p,q}$ of $V_\CC$ is the eigenspace of $\SS(\CC)= \CC^* \times
  \CC^*$ associated to the character $(z, w) \mapsto z^{-p} w^{-q}$,  $F^p= \bigoplus_{p' \geq p} V^{p',
    n-p'}$ and $V^{p,q}= F^p \cap \ol{F^{q}}$.

For example there exists a unique $R$-Hodge structure of weight $-2n$
on $V= (2 \pi i)^n R$ called the Tate-Hodge structure and denoted
$R(n)$.

A morphism $f:(V, \varphi) \lo (V', \varphi')$ of pure $R$-Hodge structures
is a morphism $f: V \lo V'$ of $R$-module such that $f_\RR: V_\RR \lo V'_\RR$ commutes with
the action of $\SS$.

We denote by $\He_R$ the category of pure $R$-Hodge structures and,
for every $n \in \ZZ$, by $\He_R^n$ the full subcategory of $\He_R$ of pure $R$-Hodge structures of weight
$n$.

If $(V, \varphi)$ is a pure $R$-Hodge structure of weight $n$, a
polarisation for $(V, \varphi)$ is a morphism of $R$-Hodge structures $Q:
V^{\otimes 2} \lto R(-n)$ such that $(2 \pi i)^n Q(x, \varphi(i)y)$
is a positive-definite bilinear form on $V_\RR$. 

\subsection{Mixed Hodge structures}

We denote by $K$ the field $R \otimes_\ZZ \QQ$.

A mixed $R$-Hodge structure ($R$MHS) is a triple $(V, W_\bullet,
F^\bullet)$ consisting of a Noetherian $R$-module $V$, a finite
ascending filtration $W_{ \bullet}$ of $V_K:= V \otimes_R K$ 
(called the weight filtration) and a finite
decreasing filtration $F^{\bullet}$ of $V_\CC$ (called the Hodge
filtration) such that for each $n
\in \ZZ$ the couple $(\Gr_n^W V_K, \Gr_n^W(F^\bullet))$ is a pure
$K$-Hodge structure of weight $n$. 

A pure $R$-Hodge structure $V$ of weight $n \in \ZZ$ is then a special
case of a mixed $R$-Hodge structure by defining the weight filtration as
$W_{n'}V_K =V_K$ for $n' \geq n$ and $W_{n'} V_K = 0$ for $n'<n$. The notions
of weight greater or smaller than $n \in \ZZ$ are defined in the
obvious way. 

We say that an $R$MHS is of type
$\varepsilon \subset \ZZ \times \ZZ$ if the Hodge numbers $h^{p,q} :=
\dim_\CC (\Gr^W_{p+q} V_K)^{p,q}$ are zero for $(p,q) \not \in
\varepsilon$ and non-zero for $(p,q) \in
\varepsilon$.

A morphism $f:(V, W_\bullet, F^\bullet) \lo (V', W'_\bullet,
F'^\bullet)$ of $R$MHS is a morphism $f: V \lo V'$ of $R$-modules
respecting the weight and Hodge filtration.

A graded polarization of a mixed $R$-Hodge structure is the datum of a
polarization on the pure $K$-Hodge structure $\Gr^W V_K$.

Following \cite[Th.2.3.5]{De2}, the category $\MH_R$ of mixed $R$-Hodge
structures is Abelian (where the kernels and cokernels of morphisms are
endowed with the induced filtrations); the functors $\Gr^W_n : \MH_R
\lo \He_K^n$ and $\Gr^p_F: V \mapsto \Gr^p_F(V_\CC)$ are exact. Moreover every morphism $f: (V, W_\bullet, F^\bullet) \lo (V',
W'_\bullet, F'^\bullet)$ of $R$MHS is strictly compatible with
the weight and Hodge filtrations (meaning that the inclusions $f(W_n)
\subset W'_n$ and $f(F^p) \subset F'^p$ satisfy $f(W_n) = f(V_K) \cap
W'_n$ and $f(F^p) = f(V_\CC) \cap F'^p$).

We wish to extend our group-theoretic description from pure Hodge
structures to mixed ones. Let $(V, W_\bullet, F^\bullet)$ be a mixed
$\RR$-Hodge structure. A splitting of $(V, W_\bullet, F^\bullet)$ is a
bigrading $V_\CC = \bigoplus_{p,q} V^{p,q}$ such that $W_n V =
\sum_{p+q \leq n} V^{p,q}$ and $F^p = \sum_{r\geq p} V^{r,s}$. Deligne proved that any mixed
$\RR$-Hodge structure admits a unique preferred splitting:

\begin{prop} (Deligne) Let $(V, W_\bullet, F^\bullet)$ be a mixed
$\RR$-Hodge structure. It admits a unique splitting $(V^{p,q})_{p,q}$
satisfying:
\begin{equation} \label{splitting} 
V^{p,q} = \ol{V^{q,p}} \mod \bigoplus_{r<p, s<q} V^{r,s}.
\end{equation}
This splitting is functorial.
\end{prop}

\begin{rem}
A mixed $\RR$-Hodge structure $(V, W_\bullet, F^\bullet)$ such that
Deligne's splitting satisfies moreover $V^{p,q} = \ol{V^{q,p}}$ is
said to be {\em split over $\RR$}. 
\end{rem}

Given a mixed $R$-Hodge structure $(V, W_\bullet, F^\bullet)$,
Deligne's splitting on $V_\RR$ defines a unique morphism $\varphi_\CC: \SS_\CC
\lo \GL(V_\CC)$ such that $V^{p,q}$ is the eigenspace for the
character $(z, w) \mapsto z^{-p} w^{-q}$ of $\SS_\CC$. When $(V,W_\bullet, F^\bullet)$ splits over $\RR$ the morphism $\varphi_\CC$ is
the complexification of a morphism $\varphi: \SS \lo \GL(V_\RR)$. In
particular we recover our initial definition of a pure $R$-Hodge
structure. 

Conversely:

\begin{prop} \cite[1.4 and 1.5]{Pink89} \label{propPink}
Let $V$ be a Noetherian $R$-module and $\varphi_\CC: \SS_\CC \lo \GL(V_\CC)$ be a group morphism. It
defines an $R$MHS on $V$ if and only if, denoting by $\PP$ the
$K$-algebraic group Zariski
closure over $K$ of $\varphi_\CC(\SS_\CC)$, by $\W$ the unipotent
radical of $\PP$ and by $\pi: \PP \lo \G:= \PP/\W$ the reductive
quotient of $\PP$, the following conditions are satisfied:
\begin{itemize}
\item[(1)] The composite $\SS_\CC \stackrel{\varphi_\CC}{\lo} \PP_\CC \stackrel{\pi}{\lo} \G_\CC$ is
  defined over $\RR$.
\item[(2)] The composite $\G_{m, \RR} \stackrel{w}{\lo} \SS
  \stackrel{\pi \circ \varphi}{\lo}
  \G_\RR$ is defined over $K$.
\item[(3)] The weight filtration of the mixed-$R$ Hodge structure on
  $\p$ defined by $\Ad_\PP \circ \varphi_\CC$ satisfies $W_{-1}(\p) = \w$, where $\w$ denotes the Lie algebra of
  the $K$-group $\W$.
\end{itemize}
\end{prop}

\begin{rem}
Notice that, by $(2)$,  $\Ad_\PP \circ \varphi_\CC$ endows $\g$ with a pure
$K$-Hodge structure of weight $0$, hence, by $(3)$, $\p$ with a
mixed $K$-Hodge structure of weight $\leq 0$. In particular $W_0 \p
= \p$.
\end{rem}

\subsection{Mumford-Tate groups}
The category of mixed $K$-Hodge structures is a $K$-linear tensor category which is rigid
and has an obvious exact faithful $K$-linear tensor functor $\omega: (V_K,
W_\bullet, F^\bullet) \mapsto V_K$.

For any $R$MHS $(V, W_\bullet, F^\bullet)$ we denote by $\langle V
\rangle$ the Tannakian subcategory of $\MH_K$
generated by $(V_K, W_\bullet, F^\bullet)$ and $\omega_V$ the
restriction of the tensor functor $\omega$ to $\langle V \rangle$; in
other words $\langle V \rangle$ is the smallest full subcategory
containing $(V_K, W_\bullet, F^\bullet)$ and the trivial $K$MHS and
stable under $\oplus$, $\otimes$, and taking subquotients. Then the
functor $\Aut^\otimes (\omega_V)$ is representable by some closed
$K$-algebraic subgroup $\MT(V)$ of $\GL(V_K)$, called the Mumford-Tate
group of $\VV$, and $\omega_V$ defines
an equivalence of categories $\langle V \rangle \simeq \Rep_K \MT(V)$
(\cite[II, 2.11]{DM}).

There are various equivalent definitions for $\MT(V)$, in particular
in terms of Hodge tensors. Recall
that a Hodge class for $V$ is any vector in $F^0V_\CC \cap
W_0V_K$. Let $T^{m, n}V_K$ denote the mixed
$K$-Hodge structure $V^{\otimes m} \otimes \HHom(V,
R)^{\otimes n} \otimes_R K$. A Hodge tensor for $V$ is a Hodge class
in some $T^{m, n}V_K$.

\begin{lem} \cite[Lemma2]{An92}
Let $(V, W_\bullet, F^\bullet)$ be a mixed $R$-Hodge structure. Then:

(a) any tensor fixed by $\PP(V)$ in some $T^{m, n}V_K$ is a Hodge
tensor. Conversely $\PP(V)$ is the stabilizer in $\GL(V_K)$ of the
Hodge tensors for $V$.

(b) $\MT(V)$ is the $K$-Zariski-closure
of the image of $\varphi$ in $\GL(V_K)$ (hence $\MT(V)$ is connected),
and if moreover $V$ is pure polarizable then $\MT(V)$ is reductive.

(b) The group $\MT(V)$ preserves $W_\bullet$ and is an extension of
the group $\MT(\Gr^W V_K)$ by a unipotent subgroup; in particular if
$V$ is graded-polarizable then $\MT(\Gr^W V_K)$ is the quotient of
$\MT(V)$ by its unipotent radical.
\end{lem}

\subsection{Variations of mixed Hodge structure} \label{KS}
Hodge theory as recalled above can be considered as the particular
case over a point of Hodge theory over an
arbitrary base in the category of complex manifolds.

Let $S$ be a complex manifold and $\OO_S$ its sheaf of holomorphic functions. A variation of mixed $R$-Hodge
structures ($R$VMHS) over
$S$ is a triple $(\VV, W_\bullet, F^\bullet)$, where:
\begin{itemize}
\item[(1)] $\VV$ is a locally constant $R_{S}$-module on $S$, 
\item[(2)] $W_\bullet$ is a finite increasing filtration (called the
  weight filtration) of the $K$-local
  system $\VV_K$ by $K$-local sub-systems,
\item[(3)] $F^\bullet$ is a finite descending filtration (called the
  Hodge filtration) of the holomorphic vector bundle $\VVV:= \VV \otimes_{R_{S}} \OO_{S}$
 by holomorphic subbundles,
\end{itemize}
such that 
\begin{itemize}
\item[(a)] for each $s \in S$, the triple $(\VV_s, (W_\bullet)_{s},
  F^\bullet_s)$ is a mixed $R$-Hodge structure.
\item[(b)] the flat connection $\nabla: \VVV \lo \VVV \otimes \Omega^1_S$
  whose sheaf of horizontal sections is $\VV_\CC$ satisfies the
  Griffiths' transversality condition
\begin{equation} \label{GT}
\nabla F^\bullet \subset \Omega^1_{S} \otimes F^{\bullet
  -1}\;\;.
\end{equation}
\end{itemize}

A graded polarization $\Psi$ for $(\VV, W_\bullet, F^\bullet)$ is a
sequence
$$ \Psi_k: \Gr^W_k (\VV_K) \times \Gr^W_k (\VV_K)  \lo K(-k)_S$$ 
of $\nabla$-flat bilinear forms inducing graded polarisations $\Psi_{k,s}$ on
the mixed $R$-Hodge structure  $(\VV_s, (W_\bullet)_{s},
  F^\bullet_s)$ for all $s \in S$.

Variations of mixed Hodge structure over a smooth quasi-projective base $S$ can have a pathological behaviour at infinity. The {\em
  admissibility} condition was defined by Kashiwara \cite{Ka} to remedy this
problem. As its statement is technical and we won't need it in this
expository note (but it will be crucial in the proofs!), we content
ourselves to refer to \cite{Ka}, \cite{StZu} and
\cite[Def. 7.2]{BrZu}. Notice that every geometric variation of mixed
Hodge structure is admissible. {\em From now on any $\ZZ$VMHS is graded-polarizable and admissible}.

An important infinitesimal invariant associated to an $R$VMHS $(\VV, W_\bullet, F^\bullet)$ over
$S$ is the Kodaira-Spencer map 
\begin{equation} \label{ks1}
\ol{\nabla}: TS^\sm \lo \Gr^{-1}_{F} (W_0 \End \VV \otimes_\QQ
\OO_S)\;\;,
\end{equation}
which is nothing but the flat connection $\nabla$ considered at first
order. Indeed the maps $\nabla: F^p  \lo F^{p-1} 
\otimes_{\OO_S} \Omega^1_S$ and  $\nabla: F^{p-1} \lo F^{p-2}
\otimes_{\OO_S} \Omega^1_S$ induces an $\OO_S$-linear morphism
$\ol{\nabla} : \Gr^p_F \VVV \lo \Gr^{p-1} \VVV \otimes_{\OO_S}
\Omega^1_S$. The alternative writing~(\ref{ks1}) follows from the fact
that $\nabla$ respects the weight filtration. 

\subsection{Hodge locus and special subvarieties} \label{VHS}
Let $R= \ZZ$ or $\QQ$. With the notations of the previous section, 
for $s$ in $S$ let $\PP_s \subset \GL(\VV_{\QQ, s})$ denote the
Mumford-Tate group of the fiber $\VV_s$. By \cite[Lemma 4]{An92}
(following Deligne in the pure case)  the group $\PP_s$ is locally
constant on the complement $S^0:= S \setminus HL(S, \VV)$ of a
countable union $\HL(S, \VV)
\subset S$ of proper irreducible analytic subvarieties of $S$. The
locus $\HL(S, \VV)$ is called the Hodge locus of $(S, \VV)$. Let $\pi_S: \tilde{S}
\lo S$ denotes a universal covering and 
choose a flat trivialization $\pi_S^* \VV_\QQ \simeq \tilde{S} \times
V$. The choice of a point $\tilde{s} \in \tilde{S}$ such that
$\pi_S(\tilde{s}) = s$ gives an identification $\VV_{\QQ, s} \simeq
V$, hence an injective homomorphism $i_{\tilde{s}}: \MT_s
\hookrightarrow \GL(V)$. For $s \in S^0$ (such a point of $S$ is
called Hodge generic) the image $\MT_{S}(\VV):= \im(i_{\tilde{s}}) \subset
\GL(V)$ neither depends on $w$ nor on the choice of $\tilde{s}$ (it is
called the generic Mumford-Tate group of $(S,V)$); for
all $s \in HL(S, \VV)$ and $\tilde{s} \in \ti{S}$ above $s$, the image
of $i_{\tilde{s}}$ is a proper subgroup of $\MT_{S}(\VV)$.

When $S$ is quasi-projective and $\VV$ is of geometric origin, the Hodge conjecture implies that
$\HL(S, \VV)$ is in fact a countable union of closed irreducible algebraic
subvarieties of $S$. The following result is fundamental in the study of
Hodge loci:
\begin{theor}[\cite{CDK95} in the pure case, \cite{BPS} in the mixed one]
Suppose that $S$ is quasi-projective. Then the Hodge locus $\HL(S,
\VV)$ is a countable union of closed irreducible {\em algebraic}
subvarieties of $S$.
\end{theor}

The irreducible components of the Zariski-closure of the strata of $\HL(S,
\VV)$ where $\PP_s$ is locally constant are called
{\em special subvarieties} of $(S, \VV)$. Special subvarieties of
dimension zero are called {\em special points} of $(S, \VV)$. Special
points whose Mumford-Tate group is a torus are called {\em CM points}
(where CM stands for Complex Multiplication).


\section{Mixed Hodge varieties} \label{Mixed Hodge varieties}
In this section we define (connected) mixed Hodge
varieties, a generalization of (connected) mixed Shimura
varieties defined by Deligne \cite{De2} and Pink \cite{Pink89}. All the
ingredients are due to Pink \cite{Pink89}.  We refer to the recent
monograph \cite{GGK} for more details in the pure case.

\subsection{Mixed Hodge data} \label{mixed Hodge data}
In this section $R = K= \QQ$. We want to parametrize mixed $\QQ$-Hodge structures, with fixed weight
filtration and given Hodge numbers, by a homogeneous space described in
terms of homomorphisms $\varphi_\CC : \SS_\CC \lo \PP_\CC$ and carrying
a canonical complex structure. The following result of Pink will guide
our definition.

\begin{prop} \label{1.7} \cite[1.7]{Pink89} 
Let $(\PP, X_\PP)$ be a pair of a connected linear
algebraic group $\PP$ over $\QQ$ and a $\PP(\RR) \W(\CC)$-conjugacy
class $X_\PP$ in $\HHom(\SS_\CC, \PP_\CC)$ (where $\W$ denotes the
unipotent radical of $\PP$). Assume that for one $\varphi \in X_\PP$
(and then for any) the conditions $(1)$, $(2)$ and $(3)$ of
\Cref{propPink} are satisfied. 

Let $\rho: \PP \lo \GL(M)$ be any faithful (finite dimensional, algebraic) representation of $\PP$. Then:

\begin{itemize}
\item[(a)] The image $\cD_{\PP, X_\PP}$ of the obvious map 
$$ \psi: X_\PP \lo \{ \textnormal{mixed}\; \QQ\textnormal{-Hodge structure on M}\}$$
admits a unique structure of complex manifold such that the Hodge
filtration on $M_\CC$ depends analytically on $\psi(\varphi)$. This
structure is $\PP(\RR) \W(\CC)$-invariant and $\W(\CC)$ acts
analytically on $\cD_{\PP, X_\PP}$.
\item[(b)] The complex manifold $\cD_{\PP, X_\PP}$ is independent of the choice
  of the faithful representation $M$. 
 \end{itemize}
\end{prop}

\Cref{1.7} is proved by noticing that, as the weight filtration is
constant on $M$ when $\varphi$ varies in $X_\PP$, the Hodge filtration
on $M$ gives an injective map of $\cD_{\PP, X_{\PP}}$ into the flag
manifold $\hat{\cD}_{\PP, X_\PP}:=\PP(\CC)/
\exp(F^0 \p_\CC)$, which is easily shown to be an open embedding.

\begin{rem}
It is important to notice that:

- the surjective map $\psi: X_\PP \lo
  \cD_{\PP, X_{\PP}}$ is not a bijection in general. This is the case if and
  only if $F^0 \w_\CC$ is trivial (see \cite[1.8 (b)]{Pink89}). 

- the group $\PP(\RR)$ does not in general act transitively on $\cD_{\PP,
    X_{\PP}}$ for the action of $\PP(\RR)$ on $\cD_{\PP,
    X_{\PP}}$ defined in \Cref{1.7}. It acts transitively on ``the set of real points of 
$\cD_{\PP, X_{\PP}}$'', namely the points for which the corresponding mixed Hodge
structure on (any) $M$ is split over $\RR$.
\end{rem}

\begin{defi} \label{mixed Hodge datum}
A pair $(\PP, X_\PP)$ as in \Cref{1.7} is called a
{\em mixed Hodge datum}. We call $\cD:= \cD_{\PP, X_\PP}$ the {\em mixed Mumford-Tate domain}
associated to the mixed Hodge datum $(\PP, X_\PP)$.
\end{defi}

\begin{defi}
A connected mixed Hodge datum is a triple $(\PP, X_\PP, \cD^+)$,
where $(\PP, X_\PP$ is a mixed Hodge datum and $\cD^+$ is a connected
component of $\cD:= \cD_{\PP, X_\PP}$.
\end{defi}

\begin{defi} \label{morphism Hodge data}
A morphism of mixed Hodge data $(\PP, X_\PP) \lo
(\PP', X_{\PP'})$ is
a homomorphism $\rho: \PP \lo \PP'$ of $\QQ$-algebraic groups which
induces maps $X_\PP \lo X_{\PP'}$. It naturally induces a holomorphic
map $\cD \lo \cD'$. A morphism of connected mixed
Hodge data $(\PP, X_\PP, \cD^+) \lo (\PP', X_{\PP'}, \cD'^+)$ is a
morphism of mixed Hodge data $(\PP, X_\PP, \cD) \lo
(\PP', X_{\PP'}, \cD')$ whose last component maps
$\cD^+$ to $\cD'^{+}$.
\end{defi}

\subsection{Mixed Hodge datum of Shimura type} \label{ShimType}

Let $(\PP, X_\PP)$ be a mixed Hodge datum and $M$ a representation of
$\PP$. The representation $M$ defines a $\PP(\RR) \W(\CC)$-equivariant local system $\M$ on $\cD$,
supporting a family of mixed Hodge structures: a flat weight
filtration $W_\bullet$ on $\M$ and a holomorphic Hodge filtration
$F^\bullet$ on $\M \otimes_\QQ \OO_{\cD}$ such that
for each $s \in \cD$, the triple $(\M_s, (W_\bullet)_{s},
  F^\bullet_s)$ is a mixed $\QQ$-Hodge structure. In most cases however,
  the Griffiths transversality condition~(\ref{GT}) is not satisfied:

\begin{prop} \cite[1.10]{Pink89} \label{1.10}
Let $(\PP, X_\PP)$ be a mixed Hodge datum and $M$ a representation of
$\PP$. The triple $(\M, W_\bullet, F^\bullet)$ is a variation of mixed
$\QQ$-Hodge structure on $\cD$ if and only if the Hodge
structure on $\p$ is of type
$$\{(-1,1), (0,0), (1, -1)\}, \{(-1, 0), (0, -1)\}, \{(-1, -1)\}\;\;.$$
\end{prop}

\begin{defi}
A mixed Hodge datum $(\PP, X_\PP)$ satisfying the condition of
\Cref{1.10} is said to be of Shimura type.
\end{defi}

Let us now relate mixed Hodge data of Shimura type with mixed Shimura
varieties in the sense of \cite[]{Pink89}
Let $(\PP, X_\PP)$ be a mixed Hodge datum of Shimura type. Let
$\UU \subset \W$ be the unique connected subgroup such that $\Lie \UU
= W_{-2} \Lie \W$. Following \cite[1.15]{Pink89} we define $\cX
\subset X_\PP$ as the subset of the $\varphi_\CC$'s defined by the following condition stronger than
(\Cref{propPink}(1)):

\begin{itemize}
\item[(1')] $\SS_\CC \stackrel{\varphi_\CC}{\lo} \PP_\CC
  \stackrel{\pi'}{\lo} (\PP/ \UU)_\CC$ is
  defined over $\RR$.
\end{itemize}

In \cite[1.16]{Pink89} Pink proves that the restriction $\psi_{|\cX}:
\cX \lo \cD$ of $\psi: X_\PP \lo \cD$ is a bijection. The triple
$(\PP, \cD, \psi_{|\cX}^{-1})$ is then a mixed Shimura datum in the
sense of \cite[2.1]{Pink89}. Conversely if $(\PP, \cX, h)$ is a mixed
Shimura datum as in \cite[2.1]{Pink89} let us define $X_\PP$ as the $\PP(\RR) \W(\CC)$-conjugacy
class of any point of $h(\cX)$. Then $(\PP, X_\PP)$ is a
mixed Hodge datum.

\subsection{Connected mixed Hodge varieties} 

\subsubsection{Definitions}
We refer to \cite{bor} and \cite[p.33, 34, 42]{Milne} for an introduction to
arithmetic groups, congruence subgroups, neat subgroups. 

Let $(\PP, X_\PP, \cD^+)$ be a connected mixed Hodge datum. The
stabilizer $\PP(\RR)^+$ of $\cD^+$ is easily seen to be open in
$\PP(\RR)$. Set $\PP(\QQ)^+:= \PP(\QQ) \cap \PP(\RR)^+$. Then any
congruence subgroup $\Gamma \subset \PP(\QQ)^+$ acts properly
discontinuously on $\cD^+$, so that $\Gamma \backslash \cD^+$ is a
complex analytic space with at most finite quotient
singularities. Every sufficiently small finite index congruence subgroup of
$\Gamma$ acts freely on $\cD^+$. Replacing $\Gamma$ by such a subgroup
if necessary, the quotient $\Gamma \backslash \cD^+$ is thus a complex
manifold and the map $\cD^+ \lo \Gamma \backslash \cD^+$ is
unramified.

\begin{defi}(connected mixed Hodge variety) \label{Hodge defi}
Let $(\PP, X_\PP, \cD^+)$ be a connected mixed Hodge datum and $\Gamma
\subset \PP(\QQ)^+$ a congruence subgroup.
The complex analytic variety $\Hod_\Gamma^0(\PP, X_\PP, \cD^+):= \Gamma \backslash \cD^+$ is called the
connected mixed Hodge variety associated with $(\PP, X_\PP, \cD^+)$ and
$\Gamma$. It is called a connected (pure) Hodge variety if $\PP$ is
reductive. The class of an element $x \in \cD^+$ is denoted $[x] \in
\Gamma \backslash \cD^+$.
\end{defi}
 
As in \cite[3.4]{Pink89} one obtains:
\begin{lem}
Let $\rho: (\PP, X_\PP, \cD^+) \lo (\PP', X_{\PP'}, \cD'^+)$ be a morphism
of connected mixed Hodge data and $\Ga \subset \PP(\QQ)^+$ and
$\Gamma' \subset \PP'(\QQ)^+$ congruence lattices such that $\rho(\Ga)
\subset \Gamma'$. Then the map
$$ [\rho]: \Gamma \backslash \cD^+ \lo \Gamma' \backslash \cD'^+$$
mapping $[x]$ to $[\rho \circ x]$ is well-defined and holomorphic.
Such a map is called a Hodge morphism of connected mixed Hodge varieties.
\end{lem}

\begin{defi}(mixed Hodge variety)
Let $(\PP, X_\PP)$ be a mixed Hodge datum and $K \subset \PP(\AAf)$ a compact open
subgroup of the finite adelic points of $\PP$. The mixed Hodge variety
$\Hod_K(\PP, X_\PP)$ is the complex analytic space $\PP(\QQ)
\backslash (\cD \times \PP(\AAf)/K)$. Here $\PP(\QQ)$ acts diagonally on $\cD$ and $\PP(\AAf)$, the group $K$ acts on
$\PP(\AAf)$ on the right, and $\PP(\AAf)$ is endowed with the adelic
topology.
\end{defi}

As in \cite[lemma 5.11, 5.12, 5.13]{Milne} one shows:
\begin{lem}
Fix $\cD^+$ a connected component of $\cD$.
Let $\cC$ be a (finite) set of representatives for the finite double coset
space $\PP(\QQ)^+ \backslash \PP(\AAf) /K$. Then we have a homeomorphism
$\Hod_K(\PP, X_\PP) = \bigsqcup_{g \in \cC} \Gamma_g \backslash \cD^+ \;,$
where $\Gamma_g$ is the congruence subgroup $g K g^{-1} \cap
\PP(\QQ)^+$ of $\PP(\QQ)^+$.

In particular the mixed Hodge variety $\Hod_K(\PP, X_\PP)$ is a finite union of connected mixed Hodge
varieties.
\end{lem}

When $(\PP, X_\PP)$ is a mixed Hodge datum of Shimura type and $(\PP, \cX, h)$ is
the associated mixed Shimura datum as in \Cref{ShimType}, the mixed
Hodge variety $\Hod_K(\PP, X_\PP)$ coincide with the mixed Shimura
variety $M^K(\PP, \cX)(\CC)$ defined in \cite[3.1]{Pink89}. We will
say that $\Hod_K(\PP, X_\PP)$ is of Shimura type and denote it by
$\Sh_K(\PP, X_\PP)$.

Let $\cD^+$ be a connected component of $\cD$ and $\Ga \subset
\PP(\QQ)^+$ a torsion-free congruence subgroup. If $M$ is a representation of
$\PP$, the family $(\M, W_\bullet, F^\bullet)$ of mixed $\QQ$-Hodge structures on $\cD^+$ descends
to the connected mixed Hodge variety $\Gamma \backslash \cD^+$.
This defines a variation of mixed $\QQ$-Hodge structure
(i.e. satisfies Griffiths' transversality) if and only if
$\Gamma \backslash \cD^+$ is a connected mixed Shimura variety.

\subsubsection{Algebraicity}
In general the connected mixed Hodge variety $\Gamma \backslash \cD^+$
and the mixed Hodge variety $\Hod_K(\PP, X_\PP)$ are complex
analytic variety which do not admit any algebraic structure. If $(\PP,
X_\PP)$ is a pure Hodge datum one can
show (\cite{GRT}) that the pure Hodge variety $\Hod_K(\PP, X_\PP)$
admits an algebraic structure if and only if it fibers holomorphically
or antiholomorphically over a (connected) Shimura
variety $\Sh_K(\PP', X_{\PP'})$. A similar result should hold in the mixed
case.

\subsection{Special subvarieties of Hodge varieties}
Let $(\PP, X_\PP, \cD^+)$ be a connected mixed Hodge datum and $Y:=
\Gamma \backslash \cD^+$ an associated connected mixed Hodge
variety. Although a representation $M$ of $\PP$ does not in general
define a variation of mixed $\QQ$-Hodge structure on $Y$, we can still
define a notion of special subvariety of $Y$ in purely group
theoretical terms.

\begin{defi} 
The image of any Hodge morphism $T \lo Y$ between connected mixed
Hodge varieties is called a special subvariety of $Y$. It is said to
be of Shimura type if $T$ is a connected mixed Shimura variety.
\end{defi}

\begin{defi}
The Hodge locus $\HL(Y)$ of the connected mixed Hodge
variety $Y$ is the union of all special subvarieties of $Y$.
\end{defi}

\subsection{Period maps}
Let $(\PP, X_\PP)$ be a mixed Hodge datum, $\cD$ the associated mixed
Mumford-Tate domain, $K \subset \PP(\AAf)$ a neat compact open
subgroup and $M$ an algebraic
representation of $\PP$. 
In general, Griffiths' transversality for $(\M, W_\bullet, F^\bullet)$
is recovered by restricting ourselves to
{\em horizontal} subvarieties of  $\Hod_{K}(\PP, X_\PP)$, which are
defined as follows. The tangent
bundle $T\cD$ is naturally $\PP(\RR) \W(\CC)$-equivariant,
associated to the representation $F^{-\infty} {\p} /F^0 {\p}$. The
Hodge filtration on $\p$ defined by the mixed Hodge structure $\varphi_\CC : \SS_\CC \lo
{\PP}_\CC$ thus defines a natural $\PP(\RR) \W(\CC)$-equivariant 
filtration $F^\bullet$ on the holomorphic tangent bundle $T\cD$. 

\begin{defi} \label{horizontal tangent}
The horizontal tangent bundle $T_h \cD$ is defined to be the holomorphic
subbundle $F^{-1} T\cD$, associated to the representation
$\Gr^{-1}_F\p$.  By equivariance it descends to a holomorphic
subbundle $T_h \Hod_{K}(\PP, X_\PP) \subset T \Hod_{K}(\PP, X_\PP)$, called the 
horizontal tangent bundle of the mixed Hodge variety $\Hod_{K}(\PP,
X_\PP)$. One denotes by $T_h \Hod_\Ga^0(\PP, X_\PP)$ its restriction
to the connected mixed Hodge variety $\Hod_\Ga^0(\PP, X_\PP) = \Gamma
\backslash \cD^+$.
\end{defi} 

\begin{rem}
The equality $T_h \Hod_{K}(\PP, X_\PP) = T \Hod_{K}(\PP,
X_\PP)$ holds if and only if $(\PP, X_\PP)$ is of Shimura type,
i.e. if $\Hod_{K}(\PP, X_\PP)$ is a mixed Shimura variety. 
\end{rem}

\begin{defi} 
Let $S$ be a complex manifold. A holomorphic map $\Phi: S
\lo \Hod_{K}(\PP, X_\PP)$ is said to be horizontal if $\Phi_*(TS) \subset T_h
\Hod_{K}(\PP, X_\PP)$. 
\end{defi}

\begin{defi}
Let $S$ be a complex manifold. A holomorphic map $f: S \lo
\Hod_K(\PP, X_\PP)$ is said to be locally liftable horizontal if
for each point $s \in S$ there exists a neighbourhood $U_s$ of $s$ in
$S$ and a commutative diagram
$$
\xymatrix{
U_s \ar[r]^{\tilde{f}} \ar[dr]_{f} & \cD \ar[d] \\
& \Hod_K(\PP, X_\PP)}
$$
such that $\tilde{f}$ is horizontal.
\end{defi}

\begin{defi} \label{periodmap}
Let $(\PP, X_\PP)$ be a mixed Hodge datum, $K \subset \PP(\AAf)$ a compact open
subgroup and $\Hod_K(\PP, X_\PP)$ the associated Hodge variety. Let $S$
be a complex manifold. A map $$\Phi: S \lo \Hod_K(\PP, X_\PP)$$ is called
a period map if it is holomorphic, locally liftable and horizontal.
\end{defi} 

If $\Phi:S \lo \Hod_{K}(\PP, X_\PP)$ is a period map the
pullback $\Phi^*(\M, W_\bullet, F^\bullet)$ is a $\QQ$VMHS on $S$. Conversely suppose $\VV \lo S$ is a
$\ZZ$VMHS over a connected complex
manifold $S$.  Fix a Hodge generic base point $s \in S$, let $\PP_S:=
\PP_S(\VV)= \PP(\VV_s)$ be the generic Mumford-Tate group of $\VV$,
$\W_S$ its unipotent radical, $\p_S$ the Lie algebra of $\PP_S$ and
$X_S$ the $\PP_S(\RR) \W_S(\CC)$-conjugacy class of
$\varphi_s: \SS_\CC \lo \GL(\VV_{s,\CC})$. The pair $(\PP_S,
X_S)$ is a mixed Hodge datum. Let $\cD_S^+$ be the connected
component of $\cD_S:= \cD_{\PP_S, X_S}$ containing the image
$\psi(\varphi_s)$. The triple $(\PP_S, X_S, \cD^+_S)$ is a connected
mixed Hodge datum. Let $\Ga:=
\PP(\QQ) \cap \GL(\VV_{s,\ZZ})$. We define the connected mixed Hodge variety
$\Hod^0(S, \VV)$ as $\Hod^0_{\Ga}(\PP_S, X_S, \cD^+_S)$. 

The Hodge filtration for $\VV
\lo S$ defines a period map
\begin{equation} \label{PeriodMap}
 \Phi_S: S \lo \Hod^0(S, \VV)\;\;.
\end{equation}

Moreover the Kodaira-Spencer map $\ol{\nabla}: TS^{0} \lo \Gr^{-1}_F W_0 \, \End \VV$ is naturally
interpreted as the composed morphism of fiber bundles
\begin{equation*}
\xymatrix{
TS \ar[r]^>>>>{d\Phi_S} \ar@/_2pc/[rr]_{\ol{\nabla}} &\Phi_S^* \;T_h
\Hod^0(S, \VV) \ar@{^(->}[r] &\Gr^{-1}_F W_0 \,\End \VV}
\;\;.
\end{equation*}

\begin{rem}
By a classical theorem of Griffiths \cite[Theor.9.6]{Grif} one can,
enlarging $S$ if necessary, assume that $\Phi_S(S)$ is a closed complex
analytic subvariety of $\Hod^0(S, \VV)$. A long-standing conjecture of
Griffiths states that the closed complex analytic horizontal
subvariety $\Phi_S(S)$ should admit a canonical structure of quasi-projective algebraic 
variety. In the pure case we refer to the work of Sommese \cite{Som} for partial
results and to \cite{LS} and \cite{GGLR}, both of which announce the
full conjecture. 
\end{rem}


\subsection{Special subvarieties and period maps}

The following proposition follows from the definitions.

\begin{prop} \label{IdentificationSpecial}
Let $\VV = (\VV_\ZZ, W_\bullet, F^\bullet)$ be a $\ZZ$VMHS  over
a smooth quasi-projective variety $S$ with associated period
map 
$$ \Phi_S: S \lo \Hod^0(S, \VV)\;\;.$$
Let $Z$ be an
irreducible subvariety of $S$. The following conditions are
equivalent:
\begin{itemize}
\item[(1)] $Z$ is a special subvariety for $(S, \VV)$.
\item[(2)] $Z$ is an irreducible component of the $\Phi_S$-preimage of
  a special subvariety of $\Hod^0(S, \VV)$.
\end{itemize}
Thus:
$$\HL(S, \VV) = \Phi_S^{-1} \left( \Phi_S(S) \cap \HL(\Hod^0(S, \VV)) \right)\;\;.$$
\end{prop}

\section{Atypicality and optimality in terms
of period maps} \label{geometric}

Let $\VV = (\VV_\ZZ, W_\bullet, F^\bullet)$ be a $\ZZ$VMHS  over
a smooth quasi-projective variety $S$ with associated period
map 
$$ \Phi_S: S \lo \Hod^0(S, \VV) \;\;.$$

\subsection{Hodge codimension and period maps}
Let us clarify geometrically the \Cref{Hcd} of Hodge codimension in terms of $\Phi_S$ and intersection theory. 

It follows from the description of $\cD_S$ that $$\dim_\CC \Gr^{-1}_F \p_S= \rk T_h \Hod^0(S, \VV)$$ and $$\rk \Imm
\ol{\nabla} = \dim_\CC \Phi_S(S) \;\;.$$
Hence: 

\begin{cor} \label{Hcd2}
The Hodge codimension $\Hcd(S, \VV)$ as defined in \Cref{Hcd} is
equal to the codimension of the tangent space at a
Hodge generic point of $\Phi_S(S)$ in the corresponding horizontal tangent space of $\Hod^0(S, \VV)$:
$$\Hcd(S, \VV) = \rk T_h \Hod^0(S, \VV)- \dim_\CC \Phi_S(S) \;\;.$$
\end{cor}

In other words, the Hodge codimension $\Hcd(S, \VV)$ of $(S, \VV)$ is the natural codimension of
$\Phi_S(S)$ in $\Hod^0(S, \VV)$ once the Griffiths' transversality condition is taken into account.

\begin{rem}
In the simple case where $S$ is a closed algebraic subvariety of a connected mixed Shimura variety $\Sh^0_K(\PP, X)$, and $\VV$
is the restriction to $S$ of a variation of mixed Hodge
structure on $\Sh^0_K(\PP, X)$ associated to an algebraic representation
of $\PP$, then $\Hcd(S, \VV)$ coincides with
what Pink in \cite{Pink05b} calls the {\em defect} of $S$ in $\Sh$: the
codimension of $S$ in its {\em special closure} (i.e. the smallest special subvariety of $\Sh$
containing $S$).
\end{rem} 

\subsection{Atypicality and period maps}
It follows that the atypicality condition~(\ref{atypicality condition}) is
better understood in terms of the period map $\Phi_S: S \lo
\Hod^0(S, \VV)$ and intersection
theory. Let $Y \subset S$ be an irreducible algebraic subvariety. Let $\PP_S$ and $\PP_Y$ be the generic
Mumford-Tate groups of $\VV$ and $\VV_{Y^{\sm}}$ respectively, with Lie
algebras $\p_S$ and $\p_Y$ and Mumford-Tate domain $\cD_S$ and
$\cD_Y$. The restriction ${\Phi_S}_{|Y}$ of the period map $\Phi_S$ to $Y$
factorises uniquely as 
$$ \xymatrix{
Y \ar[r]^>>>>>{\Phi_{Y}} \ar@/_2pc/[rr]_{{\Phi_{S}}_{|Y}} &
\Hod^0(Y,\VV_{Y^{\sm}} ) \ar@{^(->}[r]^>>>>{ i_{Y,S}} & \Hod^0(S, \VV)}
\;\;.
$$

We thus obtain the following definition, equivalent to
\Cref{atypical}:
\begin{defi}  \label{atypical2}
Let $S$ be an irreducible smooth quasi-projective variety and $\VV \lo
S$ a $\QQ$VMHS on $S$.

A subvariety $Y \subset S$ is said to be atypical for
$(S, \VV)$ if it is irreducible and 
\begin{equation} \label{atyp2}
\dim_\CC \Phi_S(S) -\dim_\CC \Phi_S(Y) < \rk T_h \Hod^0(S, \VV) - \rk T_h\Hod^0(Y,\VV_{Y^{\sm}} ) \;\;.
\end{equation}
I.e., an atypical subvariety of $(S, \VV)$ is an irreducible subvariety $Y
\subset S$ such that the subvariety $\Phi_S(Y)$ of $\Phi_S(S)$ has
an excess intersection with the Hodge locus $\HL(\Hod^0(S, \VV))
\subset \Hod^0(S, \VV)$.
\end{defi}

\begin{rem} \label{rem0}
If $Y \subset S$ is Hodge generic in $S$ the atypicality
condition~(\ref{atyp2}) is
equivalent to $\dim_\CC \Phi_S(Y) > \dim_\CC \Phi_S(S)$ hence is never
satisfied. Thus a subvariety $Y \subset S$ atypical for $\VV$ is
always contained in the Hodge locus $\HL(S,
\VV)$.
\end{rem} 

\begin{rem} \label{rem2}
It follows from \Cref{rem1} and \Cref{atypical} that if $Y \subset Y' \subset S$ if a
pair of irreducible subvarieties, if $\PP_Y= \PP_{Y'}$ and if $Y$ is
atypical for $(S, \VV)$ then $Y'$ is atypical for $(S, \VV)$.
\end{rem}

\begin{rem} \label{rem1.1}
If follows immediately from \Cref{atypical2} that any $Y \subset S$
atypical for $(S, \VV)$ is contained in a unique subvariety of $S$
atypical for $(S, \VV)$, having $\PP_Y$ as generic Mumford-Tate group,
and maximal for these properties: the
irreducible component of $\Phi_S^{-1}(\Phi_S(S) \cap \Hod^0(Y, \VV_{Y^\sm}))$ containing $Y$. In particular maximal atypical
subvarieties of $S$ are special subvarieties of $S$.
\end{rem}

\begin{rem} \label{rem1.2}
As the set of special subvarieties for $(\PP_S, X_S)$ is countable, the set
of maximal atypical subvarieties for $(S, \VV)$ is countable. In
particular it follows from \Cref{rem1.1} that $S_\atyp(\VV) \subset \HL(S,\VV) \subset S$ is also a countable union
of irreducible algebraic subvarieties of $S$.
\end{rem}

\begin{rem} \label{rem1.3}
If $f:S \lo S'$ is a surjective morphism of
smooth quasi-projective varieties and the $\ZZ$VMHS $\VV$ on $S$ is the pull-back $f^*\VV'$ of some $\ZZ$VMHS $\VV'$ on $S'$
then $S_{\atyp}(\VV) = f^{-1} S'_{\atyp}(\VV')$.
\end{rem}

\subsection{Optimality and period maps}
Similarly, the optimality condition is better understood in terms of
period maps: an irreducible subvariety $Y \subset S$ is optimal for
$(S, \VV)$ if it is not strictly contained in a better approximation to a special
subvariety of $\Hod^0(S, \VV)$.

\subsection{Examples}

\begin{Example} \label{ex1}
Let $S$ be a Shimura variety and $\VV$ a standard variation of Hodge
structure on $S$. Then $Y \subset S$ is atypical if and only if
$\dim_\CC Y > \rk T_h\Hod^0(Y, \VV_{|Y^\sm})$, which is impossible. Hence $S_\atyp =\emptyset$.
\end{Example}

\begin{Example} (Atypical special subvarieties of Shimura type with
  dominant period map) \label{ex2}
\begin{lem} \label{atypical of Shimura type}
Let $\VV \lo S$ be a $\ZZ$VMHS over a
smooth quasi-projective variety $S$. Let $Y \subset S$ be a
special subvariety for $(S, \VV)$, of Shimura type and such that the
period map
$\Phi_Y:Y \lo \Sh^0(Y)$ is dominant. Then $Y$ is not atypical for $(S, \VV)$ if
and only if $(\PP_S, \cD_S)$ is of Shimura type, the period map $\Phi_S: S \lo
\Sh^0(S, \VV)$ is dominant, and $\VV$ is the restriction to $S$ of a standard
variation of mixed Hodge structure on $\Sh^0(S, \VV)$.
\end{lem}

\begin{proof}
Let $\Phi_S: S \lo \Hod^0(S, \VV)$ be the period map for $(S,
\VV)$. As $Y$ is special of Shimura type with dominant period map, $\rk T_h\Hod^0(Y, \VV_{Y^\sm})= \dim_\CC
\Phi_S(Y)$ and the atypicality condition~(\ref{atyp2}) for $Y$ reads
as: 
\begin{equation} \label{CM}
\dim_\CC \Phi_S(S)  < \rk T_h \Hod^0(S, \VV) \;\;.
\end{equation}

The variety $\Phi_S(S)$ is an horizontal subvariety of $\Hod^0(S, \VV)$ hence the inequality $\dim_\CC \Phi_S(S)  \leq \rk T_h
\Hod^0(S, \VV) $ always holds. It follows that the special
subvariety $Y$ is not atypical for $(S, \VV)$ if and only if $\dim_\CC \Phi_S(S)
= \rk T_h \Hod^0(S, \VV)$. This implies that the horizontal
distribution $T_h \Hod^0(S, \VV)$ of $T\Hod^0(S, \VV)$ is
integrable. 

Suppose first that $\VV$ is pure. In that case standard Lie theoretic considerations imply that $\cD_S^+$ is a Hermitian
symmetric domain. The existence of one algebraic leaf (namely $S$) for
the foliation on $\Hod^0(S, \VV)$ defined by
$T_h \Hod^0(S, \VV)$ implies that $T_h
\Hod^0(S, \VV) = T \Hod^0(S, \VV)$ (i.e. $\VV$ is of unconstrained
type in the terminology of \cite[p.12]{GGK}). Thus $\VV$ comes from a
$\ZZ$VMHS on $\Hod^0(S, \VV)$ hence $(\PP_S,
X_S)$ is of Shimura type by \Cref{1.10}. The same conclusion holds in
the mixed case by a classical devissage to the pure case.

The facts that $\Phi_S: S \lo
\Sh^0(S, \VV)$ is dominant and $\VV$ is the restriction to $S$ of a standard
variation of mixed Hodge structure on $\Sh^0(S, \VV)$ are now clear.
\end{proof}

Notice the following special case of \Cref{atypical of Shimura type}:
\begin{lem} \label{atypical CM point}
Let $\VV \lo S$ be a $\ZZ$VMHS over a
smooth quasi-projective variety $S$. Let $x \in S$ be a
CM-point. Then $x$ is not atypical for $(S, \VV)$ if
and only if $(\PP_S, X_S)$ is of Shimura type and the period map $\Phi_S: S \lo
\Sh^0(S, \VV)$ is dominant.
\end{lem}

\end{Example}

\section{On the main conjecture and its corollaries} \label{On main}

\begin{prop} \label{equivalence}
\Cref{main conjecture1}, \Cref{main conjecture2}, \Cref{main
  conjecture3} and \Cref{main conjecture4} are equivalent.
\end{prop}

\begin{proof}
Obviously \Cref{main conjecture1} implies \Cref{main conjecture2},
which in turns implies \Cref{main conjecture3}. Let us show that
\Cref{main conjecture3} implies \Cref{main conjecture1}.
So let $S$ be an irreducible smooth quasi-projective variety and $\VV
\lo S$ a $\ZZ$VMHS on $S$. Let $Y \subset S$ be an
irreducible component of the Zariski-closure of
$S_{\atyp}(\VV)$. It follows from \Cref{rem1.1} that $Y$ is the
Zariski-closure of the union $\bigcup_{i \in I} X_i$, where $I$ is a
countable set and the $X_i$ are pairwise distinct maximal atypical special subvarieties for
$(S, \VV)$. There exists an index $i \in I$ such that the variety
$X_i$ is not atypical for $(Y, \VV_{|Y})$: otherwise $Y_\atyp(\VV_{|Y})$ would be Zariski-dense
in $Y$, contradicting \Cref{main conjecture3}. Then
\begin{equation} \label{not atyp}
\Hcd(Y, \VV_{|Y^\sm}) = \Hcd(X_{i}, \VV_{|X_{i}^\sm}) < \Hcd(S, \VV),
\end{equation}
hence $Y$ is atypical in $S$. By maximality of the $X_i$'s, it follows
that the set $I$ contains only one element $\{i\}$. Thus $Y = X_i$ is a
special subvariety of $(S, \VV)$ and the result.

It remains to show that \Cref{main conjecture1} and \Cref{main
  conjecture4} are equivalent. Assume that \Cref{main conjecture4}
holds true. As maximal atypical subvarieties for $(S, \VV)$ are
optimal for $(S, \VV)$ \Cref{main conjecture1} follows. Conversely assume that \Cref{main conjecture1}
holds true and let us prove \Cref{main conjecture4} by induction on $s
= \dim_\CC S$. For $s=1$ the variety $S$ is a curve, optimal
subvarieties of $S$ are points and coincide with (maximal) atypical subvarieties of
$S$, hence the result in this case. Suppose \Cref{main conjecture4}
holds true for varieties of dimension at most $s-1$ and suppose $S$ is
of dimension $s$. Let $Z_1, \cdots, Z_n$ be the finite collection of
maximal atypical subvarieties for $(S, \VV)$. Let $(Y_i)_{i \in I}$ be a set of pairwise distinct optimal
subvarieties for $(S, \VV)$. For each $i \in I$ the optimal subvariety $Y_i
\subset S$ is atypical, hence there exists $\phi(i)  \in \{1, \cdots ,
n\}$ such that $Y_i \subset Z_{\phi(i)}$. Moreover by the very
definition of optimality either $Y_i = Z_{\phi(i)}$ or $Y_i$ is
optimal for $(Z_{\phi(i)}, \VV_{|Z_{\phi(i)}^\sm})$. As $\dim_\CC
Z_{\phi(i)} <s$ it follows by the induction hypothesis that the set
$I$ is finite. Hence the result.
\end{proof}

\subsection{\Cref{main conjecture1} as a generalized Zilber-Pink
  conjecture}

Notice that the restriction of \Cref{main conjecture1} to the class of
pairs $(S, \VV)$ where $S$ is a subvariety of a Shimura variety
$\Sh^0(S, \VV)$ and $\VV$ is the restriction to $S$
of a standard $\ZZ$VMHS on $\Sh^0(S, \VV)$ is the
Zilber-Pink conjecture as stated for example by Pila
(\cite[Conj.2.3]{Pila}) (this conjecture was formulated by Zilber
\cite{Zil02} in the case of multiplicative groups, it is a stronger version of the original
Pink Conjecture \cite[Conj.1.1]{Pink05b}).

\subsection{\Cref{main conjecture1} and CM-points: a generalized
  Andr\'e-Oort conjecture}

\Cref{main conjecture1} immediately implies:

\begin{conj}[Andr\'e-Oort conjecture for $\ZZ$VHS, version 1] \label{conj1}

Let $\VV \lo S$ be a $\ZZ$VMHS over a
smooth irreducible quasi-projective variety $S$.
Suppose that the union $\SL(S, \VV)$ of special subvarieties for $(S, \VV)$, which
are of Shimura type with dominant period map, is Zariski-dense in
$S$. Then the Hodge datum $(\PP_S, X_S)$ is of Shimura type and the
period map $ \Phi_S: S \lo \Sh^0(S, \VV)$
is dominant.
\end{conj}

Let us show how to deduce \Cref{conj1} from our main \Cref{main
  conjecture1}.  Let $\Phi_S: S \lo \Hod^0(S, \VV)$ be the period map for
$\VV$. By assumption the special subvarieties of Shimura type and
dominant period map are Zariski-dense in $S$. It
thus follows from \Cref{main conjecture3} and Noetherian induction that there exists 
(a Zariski-dense set of) not atypical special subvarieties of Shimura
type with dominant period map
in $S$. The conclusion
follows from the \Cref{atypical of Shimura type}.

Notice that \Cref{conj1} immediately implies, and, by density of
CM-points in Shimura varieties, is in fact
equivalent to, the
following:

\begin{conj}[Andr\'e-Oort conjecture for $\ZZ$VHS, version 2] \label{conj2}

Let $\VV \lo S$ be a $\ZZ$VMHS over a
smooth irreducible quasi-projective variety $S$.
Suppose that the union of CM-points for $(S, \VV)$ is Zariski-dense in
$S$.  Then $(\PP_S, X_S)$ is of Shimura type and the period map
$\Phi_S: S \lo \Sh^0(S, \VV)$ is dominant.
\end{conj}

\begin{rem}
When $S$ is moreover defined over $\ol{\QQ}$, \Cref{conj2}
was originally stated (in a slightly different way) by
Green-Griffiths-Kerr: see \cite[Conj. VIII.B.1
p.275]{GGK}.
\end{rem}

Notice that \Cref{conj1} is equivalent to the union of the following two
conjectures:

\begin{conj}[classical Andr\'e-Oort]  \label{AO}
Let $Y \subset \Sh_K^0(\PP, X)$ be a closed irreducible algebraic
subvariety of a connected mixed Shimura variety. 
If $Y$ contains a Zariski-dense set of CM-points then $Y$ is a special
subvariety of $\Sh_K(\PP, X)$.
\end{conj}

\begin{conj} \label{conj1'}
Let $\VV \lo S$ be a $\ZZ$VMHS over a
smooth irreducible quasi-projective variety $S$.
Suppose that the set of CM-points for $(S, \VV)$ is Zariski-dense in
$S$. 
Then $(\PP_S, X_S)$ is of Shimura type.
\end{conj}

Many works have been devoted to the Andr\'e-Oort \Cref{conj1}, culminating to its proof when $\Sh_K^0(\PP, X)$ is of Abelian
type. We refer to \cite{Survey} for a detailed analysis of the
Andr\'e-Oort conjecture and references to the related works. On the other hand \Cref{conj1'} is completely open. 

The proof of the classical Andr\'e-Oort Conjecture~\ref{AO} relies
on two different ingredients: on the one hand a precise
analysis of the arithmetic of Galois orbits of CM-points (lower bound and heights),
on the other hand a geometric analysis of the distribution in
$\Sh_K(\PP, \cD)$ of {\em positive dimensional special
subvarieties}. If $(S, \VV)$ is not of Shimura type, the arithmetic of
CM-points for $(S, \VV)$ seems difficult to understand. On the
other hand it seems worth focusing on positive dimensional special
subvarieties, namely the following
geometric part of \Cref{conj1}:

\begin{conj} \label{conj3}
Let $\VV \lo S$ be a $\ZZ$VMHS over a
smooth irreducible quasi-projective variety $S$.
Suppose that the union $\SL^{>0}(S, \VV)$ of {\em positive
  dimensional} special subvarieties for $(S, \VV)$, which
are of Shimura type with dominant period map, is Zariski-dense in
$S$. Then $(\PP_S, X_S)$ is of Shimura type and the period map
$\Phi_S: S \lo \Sh^0(S, \VV)$ is dominant. 
\end{conj}

One could even ask the following question, more general than \Cref{conj1}:
\begin{Question} \label{Q3}
Let $\VV \lo S$ be a $\ZZ$VMHS over a
smooth irreducible quasi-projective variety $S$.
Suppose that the union of special subvarieties for $(S, \VV)$ which
are of Shimura type (but not necessarily with dominant period maps) is Zariski-dense in
$S$. Is it true that necessarily $(\PP_S, X_S)$ is of Shimura type?
\end{Question}

\section{Example: Calabi-Yau 3-folds} \label{CY}

Let us describe the first non-trivial incarnation of \Cref{conj1}
outside of the world of Shimura varieties. Let $X$ be a smooth
projective Calabi-Yau threefold (i.e. the canonical bundle $K_X$ is
trivial and $X$ has trivial fundamental group). Let $H_\ZZ:= H^3(X,
\ZZ)$ with its natural polarized weight $3$ $\ZZ$-Hodge structure $H_\CC = H^{3,0}
\oplus H^{2,1} \oplus H^{1,2} \oplus H^{0,3}$. One can endow $H_\ZZ$
with two different weight one Hodge structures: 

- the Weil Hodge structure $H^1_{W}$ for which $H^{1,0}_{W} := H^{0,3} \oplus
H^{2,1}$. 

- the Griffiths Hodge structure $H^1_{G}$ for which $H^{1,0}_{G} :=
H^{3,0} \oplus H^{2,1}$.

These two weight one Hodge structures define two complex structures on
the torus $H_\RR/H_\ZZ$, the Weil intermediate Jacobian $J(X)_W \coloneqq H_\ZZ
\backslash H_\CC / H^1_{W}$ which is an Abelian variety but does not
vary holomorphically with $X$, and the Griffiths intermediate Jacobian
$J(X)_G\coloneqq H_\ZZ \backslash H_\CC / H^1_{G} $ which is a mere complex
torus but varies holomorphically with $X$. In \cite{Borcea} Borcea
proves that $H_\ZZ$ has CM if and only if both the Hodge structures 
$H^1_{W} $ and  $H^1_ G$ have CM and their Mumford-Tate tori mutually commute
in $\GL(H_\QQ)$. Let $S = \Def(X)$ be the family of Calabi-Yau
threefolds deformation space of $X$ and $\VV$ be the corresponding polarized weight
$3$ variation of $\ZZ$-Hodge structure on $S^\sm$ with fiber $H_\ZZ$
at $X$. One can choose $X$ so that $S$
contains infinitely many CM-points. To the best of my knowledge, in
all examples (see \cite{VZ}, \cite{Z}, \cite{Rohde}) the irreducible subvarieties of $S$ containing a
Zariski-dense set of CM-points and maximal for these properties are of
Shimura type, as predicted by \Cref{conj1}. On the other hand 
it is not clear to me that there are only finitely many such
subvarieties as predicted by \Cref{main conjecture1}. Notice that a
weaker version of \Cref{conj1} in this case (and more generally for
Calabi-Yau $n$-folds) already appears in \cite{GuVa}.

\section{Functional transcendence} \label{functional transcendence}

One main tool for attacking \Cref{main conjecture1} or
\Cref{conj1} consists in establishing functional transcendence
statements for the period map $\Phi_S$. It is adapted from the Pila-Zannier strategy for proving the
Andr\'e-Oort \Cref{AO}, hopefully using o-minimal techniques. We refer
once more to \cite{Survey} for a description of this strategy in the case of the
Andr\'e-Oort conjecture and focus here on the expected statements.

\subsection{Weakly special subvarieties}
These functional transcendence statements detect {\em weakly special subvarieties}, a generalisation
of special subvarieties.

\begin{defi}
Consider any Hodge morphisms $R \stackrel{\pi}{\leftarrow} T \stackrel{i}{\rightarrow} Y$ between
connected mixed Hodge varieties and any point $r \in R$. Then any
irreducible component of $i(\pi^{-1}(r))$ is called a weakly special
subvariety of $Y$.

Let $\VV \lo S$ be a $\ZZ$VMHS over a
smooth quasi-projective base $S$ with associated period map $\Phi_S:S
\lo \Hod^0(S, \VV)$. Any irreducible component of
$\Phi_S^{-1}(\Phi_S(S) \cap Z)$, where $Z$ is a weakly special
subvariety of the connected mixed Hodge variety $\Hod^0(S, \VV)$ is
called a weekly special subvariety for $(S, \VV)$.
\end{defi}

In particular (taking $R$ a point) special subvarieties of $(S, \VV)$ are weakly special
for $(S, \VV)$.

\subsection{Bi-algebraic geometry}
The format for the functional transcendence statements we are
interested in is the notion of bi-algebraic structure:

\begin{defi} \label{bialg}
A bi-algebraic structure on a connected algebraic variety $S$
is a pair 
$$ (D: \ti{S} \lo X, \quad h: \pi_1(S) \lo \Aut(X))\;\;$$ 
where $\pi:\ti{S} \lo S$ denotes the universal cover of $S$, $X$ is an
algebraic variety, $\Aut(X)$ its group of algebraic
automorphisms, $h: \pi_1(S) \lo \Aut(X)$ is a group morphism
(called the holonomy representation) and 
$D$ is an $h$-equivariant holomorphic map (called the developing map).
\end{defi}

\begin{defi} \label{bialgebraic}
Let $S$ be a connected algebraic variety endowed with a
bi-algebraic structure $(D, h)$.
\begin{itemize}
\item[(i)]
An irreducible analytic subvariety $Y \subset \ti{S}$ is said to be an
irreducible algebraic subvariety of $\ti{S}$ if $Y$ is an analytic
irreducible component of $D^{-1}(\overline{D(Y)}^\Zar)$ (where
$\overline{D(Y)}^\Zar$ denotes the Zariski-closure of $D(Y)$ in $X$).
\item[(ii)]
An irreducible algebraic subvariety $Y\subset \ti{S}$, resp. $W
\subset S$, is said to be {\em bi-algebraic} if $\pi(Y)$ is an
algebraic subvariety of $S$, resp. any (equivalently one) analytic irreducible component of $\pi^{-1}(W)$ is an
irreducible algebraic subvariety of $\ti{S}$. 
\end{itemize}
\end{defi}

If $(S, \VV)$ is a $\ZZ$VMHS with lifted period map $\tilde{\Phi}_S: \ti{S} \lo
\cD_S^+$, we denote by $$\hat{\Phi}_S : \ti{S} \lo
\hat{\cD}_S$$ the composite $j_S \circ \tilde{\Phi}_S$ where $j_S: \cD_S
\hookrightarrow \hat{\cD}_S$ denotes the open embedding of the
Mumford-Tate domain $\cD_S$ in its compact dual (see \Cref{mixed
  Hodge data}). The pair $(\hat{\Phi}_S: \ti{S} \lo
\hat{\cD}_S, \rho_S: \pi_1(S) \lo \PP_S(\CC) \lo
\Aut(\hat{\cD}_S)(\CC))$ defines a natural bi-algebraic structure
on $S$, called the bi-algebraic structure of $(S, \VV)$. The relation
between the bi-algebraic structure of $(S, \VV)$ and Hodge theory in
the pure case is
given by the following proposition, whose proof will be provided in a
sequel to this note. It is proven by Ullmo-Yafaev \cite{UY1}
in the case where $S$ is a Shimura variety, and in some special case
by Friedman and Laza \cite{FL15}:

\begin{prop} \label{bi-alg}
Let $(S, \VV)$ be a $\ZZ$VHS. The bi-algebraic subvarieties of $(S, V)$ are
the weakly special subvarieties of $(S, \VV)$.
\end{prop}

A similar result should hold in the mixed case (see
\cite[Cor.8.3]{Gao} for the case of $\ZZ$VMHS of Shimura type).

\subsection{Ax-Schanuel for variation of mixed $\ZZ$-Hodge structures}
The main functional transcendence conjecture in this setting is:
\begin{conj}(Ax-Schanuel for $\ZZ$VMHS) \label{Ax-Schanuel}
Let $(S, \VV)$ be a $\ZZ$VMHS. Let $U \subset \tilde{S} \times S$ be
an algebraic subvariety and let $W$ be an irreducible component of $U
\cap \Delta$ (where $\Delta$ denotes the graph of $\pi: \tilde{S} \lo
S$).
Then $\cd_U W \geq \dim \overline{W}^{ws}$, where
$\overline{W}^{ws}$ denotes the smallest weakly special special
subvariety of $S$ containing $\pi(W)$.
\end{conj}

When applied to a subvariety $U \subset \tilde{S} \times S$ of the
form $Y \times \overline{\pi(Y)}^\Zar$ for $Y \subset \tilde{S}$
algebraic, \Cref{Ax-Schanuel} specializes
to the following:

\begin{conj}(Ax-Lindemann for $\ZZ$VMHS) \label{Ax-Lindemann}
Let $(S, \VV)$ be a $\ZZ$VMHS. Let $Y\subset \tilde{S}$ be an
algebraic subvariety. Then $\overline{\pi(Y)}^\Zar$ is a bi-algebraic
subvariety of $(S, \VV)$, i.e. weakly special for $(S, \VV)$.
\end{conj}

Establishing \Cref{Ax-Lindemann} is a crucial step in establishing
\Cref{conj3}. 







\sspace
\noindent Bruno Klingler : Universit\'e Paris-Diderot (Institut de
Math\'ematiques de Jussieu-PRG, Paris)

\noindent email : \texttt{bruno.klingler@imj-prg.fr}.

\end{document}